\documentclass[11pt]{amsart}
\usepackage{amssymb}
\usepackage{amsmath}
\usepackage[active]{srcltx}
\usepackage{t1enc}
\usepackage[latin2]{inputenc}
\usepackage{verbatim}
\usepackage{amsmath,amsfonts,amssymb,amsthm}
\usepackage[mathcal]{eucal}
\usepackage{enumerate}
\usepackage[centertags]{amsmath}
\usepackage{graphics}

\newtheorem{theorem}{Theorem}

\newtheorem{lemma}{Lemma}
\newtheorem{remark}{Remark}

\newtheorem{corollary}{Corollary}

\newtheorem*{theoremW}{Theorem W}
\email{blahota@nyf.hu}
\email{nkaroly@nyf.hu}
\email{larserik@ltu.se}
\email{giorgitephnadze@gmail.com}
\thanks{Research was supported by project TÁMOP-4.2.2.A-11/1/KONV-2012-0051,
by Shota Rustaveli National Science Foundation grants DI/9/5-100/13, DO/24/5-100/14, YS15-2.1.1-47  and by a Swedish Institute scholarship no. 10374-2015.}

\begin{document}
\author{I. Blahota, K. Nagy, L.E. Persson and G. Tephnadze}
\title[partial sums]{A sharp boundedness result for restricted maximal
operators of Vilenkin-Fourier series on martingale Hardy spaces}
\address{I. Blahota, Institute of Mathematics and Computer Sciences, College
of Ny\'\i regyh\'aza, P.O. Box 166, Ny\'\i regyh\'aza, H-4400, Hungary.}
\address{K. Nagy, Institute of Mathematics and Computer Sciences, College of
Ny\'\i regyh\'aza, P.O. Box 166, Ny\'\i r\-egyh\'aza, H-4400 Hungary,}
\address{L.E. Persson, Department of Engineering Sciences and Mathematics,
Lule\aa\ University of Technology, SE-971 87 Lule\aa, Sweden and UiT The Arctic University of Norway, P.O. Box 385, N-8505, Narvik, Norway.}
\address{G. Tephnadze, Department of Mathematics, Faculty of Exact and
Natural Sciences, Ivane Javakhishvili Tbilisi State University, Chavchavadze
str. 1, Tbilisi 0128, Georgia, \&
University of Georgia, IV, 77a Merab Kostava St, Tbilisi, 0128,  Georgia, \& Department of Engineering Sciences and Mathematics, Lule\aa University of Technology, SE-971 87 Lule\aa, Sweden.}
\date{}

\begin{abstract}
The restricted maximal operators of partial sums with respect to bounded
Vilenkin systems are investigated. We derive the maximal subspace of
positive numbers, for which this operator is bounded from the Hardy space $%
H_{p}$ to the Lebesgue space $L_{p}$ for all $0<p\leq 1.$ We also prove that
the result is sharp in a particular sense.
\end{abstract}

\maketitle



\section{ Introduction}

Pointwise convergence problems are of fundamental importance in Fourier
analysis, and as it is well known they are closely related to studying
boundedness of associated maximal operators. In the present paper we will
deal with maximal operators. Let
us first recall in brief a historical development of this theory.

It is well-known (for details see e.g. \cite{AVD} and \cite{gol}) that
Vilenkin systems do not form bases in the space $L_{1}\left( G_{m}\right) .$
Moreover, (for details see e.g. \cite{We1,We2}) there is a function in the
martingale Hardy space $H_{1}\left( G_{m}\right) ,$ such that the partial
sums of $f$ are not bounded in $L_{1}\left( G_{m}\right) $-norm, but Watari \cite{Wat} (see also Gosselin \cite{goles} and Young \cite{Yo}) proved that there exist absolute constants $ c $ and $ c_p $ such that, for $ n =1,2,..., $
\begin{eqnarray*}
\left\Vert S_{n}f\right\Vert
_{p}&\leq & c_p\left\Vert f\right\Vert
_{p}, \ \ f\in L_p(G_m), \ \ 1<p<\infty, \\ 
\sup_{\lambda >0} \lambda \mu
\left( \vert S_n f\vert>\lambda \right)&\leq& c\left\Vert f\right\Vert_{1}, \ \ \ f\in L_1(G_m), \ \ \lambda>0. 
\end{eqnarray*}
In \cite{tep7} it was proved that there exists a martingale $f\in H_{p}\left(
G_{m}\right) \left( 0<p<1\right) ,$ such that 
\begin{equation*}
\underset{n\in \mathbb{N}}{\sup }\left\Vert S_{M_{n}+1}f\right\Vert
_{L_{p,\infty }}=\infty .
\end{equation*}
The reason of divergence of $S_{M_{n}+1}f$ is that the Fourier coefficients
of $f\in H_{p}\left( G_{m}\right) $ ($0<p<1$) are not bounded (see Tephnadze 
\cite{tep2}).

Uniform and point-wise convergence and some approximation properties of the
partial sums in $L_{1}\left( G_{m}\right) $ norms were investigated by
Goginava \cite{gog1,gog2} and Avdispahi\'c, Memi\'c \cite{am}. Some related
results can also be found in the recent PhD thesis by Tephnadze \cite%
{tepthesis}. Moreover, Fine \cite{fi} obtained sufficient condition for the
uniform convergence it is in complete analogy with the Dini-Lipschitz
condition. Guličev \cite{9} estimated the rate of uniform convergence of a
Walsh-Fourier series using Lebesgue constants and modulus of continuity.
Uniform convergence of subsequences of partial sums was also investigated by
Goginava and Tkebuchava \cite{gt}. This problem was considered for the
Vilenkin group $G_{m}$ by Fridli \cite{4}, Blahota \cite{2} and G\'at \cite%
{5}.

In \cite{tep7} the following maximal operator was considered: 
\begin{equation*}
\widetilde{S}_{p}^{\ast }f:=\sup_{n\in \mathbb{N}}\frac{\left\vert
S_{n}f\right\vert }{\left( n+1\right) ^{1/p-1}\log ^{\left[ p\right] }\left(
n+1\right) },\quad 0<p\leq 1
\end{equation*}
(where $\left[ x\right] $ denotes integer part of $x$). It was proved that
the maximal operator $\widetilde{S}_{p}^{\ast }$ is bounded from the Hardy
space $H_{p}\left( G_{m}\right)$ to the space $L_{p}\left( G_{m}\right) .$
Moreover, if $0<p\leq 1$ and $\varphi :\mathbb{N}_{+}\rightarrow \lbrack
1,\infty )$ is a non-decreasing function satisfying the condition

\begin{equation*}
\overline{\lim_{n\rightarrow \infty }}\frac{\left( n+1\right) ^{1/p-1}\log ^{%
\left[ p\right] }\left( n+1\right) }{\varphi \left( n\right) }=+\infty ,
\end{equation*}%
then 
\begin{equation*}
\sup_{n\in \mathbb{N}}\left\Vert \frac{S_{n}f}{\varphi \left( n\right) }%
\right\Vert _{L_{p,\infty }\left( G_{m}\right) }=\infty ,\text{ for }0<p<1,
\end{equation*}%
and%
\begin{equation*}
\sup_{n\in \mathbb{N}}\left\Vert \frac{S_{n}f}{\varphi \left( n\right) }%
\right\Vert _{1}=\infty .
\end{equation*}

It is also known (for details see e.g. Weisz \cite{We2}) that 
\begin{equation*}
\left\Vert S_{n_{k}}f\right\Vert _{1}\leq c\left\Vert f\right\Vert _{1}
\end{equation*}%
holds if and only if 
\begin{equation*}
\sup_{k\in \mathbb{N}}\left\Vert D_{n_{k}}\right\Vert _{1}<c<\infty ,
\end{equation*}%
where $D_{n_{k}}$ denotes the $n_{k}$th Dirichlet kernel with respect to
Vilenkin system. Moreover, the corresponding subsequence $S_{n_k}$ of the
partial sums $S_{{n}}$ are bounded from the Hardy space $H_{p}\left(
G_{m}\right) $ to the Hardy space $H_{p}\left( G_{m}\right) ,$ for all $p>0.$

It is also well-known (for details see e.g. Weisz \cite{We2} and Tephnadze 
\cite{tepthesis}) that the following restricted maximal operator 
\begin{equation*}
{S}^{\#}f:=\sup_{n\in \mathbb{N}}\left\vert S_{M_{n}}f\right\vert
\end{equation*}%
is bounded from the martingale Hardy space $H_{p}\left( G_{m}\right) $ to
the Lebesgue space $L_{p}\left( G_{m}\right) ,$ for all $p>0$.

In this paper we find the maximal subspace of positive numbers, for which
the restricted maximal operator of partial sums with respect to Vilenkin
systems in this subspace is bounded from the Hardy space $H_{p}$ to the Lebesgue space $L_{p}$ for all $0<p\leq 1.$ As applications, both some well-known and
new results are pointed out.

The paper is organized as follows: Some Preliminaries (definitions,
notations and basic facts) are presented in Section 2. The main result
(Theorem \ref{theorem4.4}) and some of its consequences (Corollaries \ref%
{corollary4.45}-\ref{corollaryae9}) are presented and discussed in Section
3. Theorem \ref{theorem4.4} is proved in Section 5. For this proof we need
some Lemmas, one of them is new and of independent interest (see Section 4).

\section{Preliminaries}

Let $\mathbb{N}_{+}$ denote the set of the positive integers, $\mathbb{N}:=%
\mathbb{N}_{+}\cup \{0\}$ and assume that $m:=(m_{0},m_{1},\ldots)$ is a
sequence of positive integers not less than 2.

Denote by 
\begin{equation*}
Z_{m_{k}}:=\{0,1,\ldots,m_{k}-1\}
\end{equation*}%
the additive group of integers modulo $m_{k}.$

Define the group $G_{m}$ as the complete direct product of the group $%
Z_{m_{k}}$ with the product of the discrete topologies of $Z_{m_{k}}$`s.

The product measure $\mu $ of the measures 
\begin{equation*}
\mu _{k}\left( \{j\}\right) :=1/m_{k} \quad (j\in Z_{m_{k}})
\end{equation*}
is a Haar measure on $G_{m}$ with $\mu \left( G_{m}\right) =1$.

If the sequence $m:=(m_{0},m_{1},\ldots)$ is bounded, then $G_{m}$ is called
a bounded Vilenkin group, otherwise it is called an unbounded one. In the
present paper we deal only with bounded Vilenkin groups.

The elements of $G_{m}$ are represented by sequences 
\begin{equation*}
x:=(x_{0},x_{1},\ldots,x_{k},\ldots)\quad \left( x_{k}\in Z_{m_{k}}\right) .
\end{equation*}

A base for the neighbourhood of $G_{m}$ can be given as follows: 
\begin{equation*}
I_{0}\left( x\right) :=G_{m},
\end{equation*}%
\begin{equation*}
I_{n}(x):=\{y\in G_{m}\mid y_{0}=x_{0},\ldots,y_{n-1}=x_{n-1}\}\quad (x\in
G_{m},\ n\in \mathbb{N}).
\end{equation*}%
Denote $I_{n}:=I_{n}\left( 0\right) $ for $n\in \mathbb{N}$ and $\overline{%
I_{n}}:=G_{m}\backslash I_{n}$.

It is evident that 
\begin{equation}  \label{1.1}
\overline{I_{N}}=\overset{N-1}{\underset{s=0}{\bigcup }}I_{s}\backslash
I_{s+1}.
\end{equation}

The generalized number system based on $m$ is defined in the following way 
\begin{equation*}
M_{0}:=1,{\ \quad }M_{k+1}:=m_{k}M_{k} \quad (k\in \mathbb{N}),
\end{equation*}%
Every $n\in \mathbb{N}$ can be uniquely expressed as%
\begin{equation*}
n=\sum_{j=0}^\infty n_{j}M_{j},\quad\text{ where }\quad n_{j}\in Z_{m_{j}}\
(j\in \mathbb{N})
\end{equation*}
and only a finite number of $n_{j}$`s differ from zero.

Let 
\begin{equation*}
\left\langle n\right\rangle :=\min \{j\in \mathbb{N}:n_{j}\neq 0\}\text{ \
and \ }\left\vert n\right\vert :=\max \{j\in \mathbb{N}:n_{j}\neq 0\},
\end{equation*}%
that is $M_{\left\vert n\right\vert }\leq n\leq M_{\left\vert n\right\vert
+1}.$ Set 
\begin{equation*}
\rho \left( n\right) :=\left\vert n\right\vert -\left\langle n\right\rangle ,%
\text{ \ for \ all \ \ }n\in \mathbb{N}.
\end{equation*}

The norm (or quasi-norm) of the space $L_{p}(G_{m})$ is defined by 
\begin{equation*}
\left\Vert f\right\Vert _{p}:=\left( \int_{G_{m}}\left\vert f\right\vert
^{p}d\mu \right) ^{1/p}\quad \left( 0<p<\infty \right) .
\end{equation*}%
The space $L_{p,\infty }\left( G_{m}\right) $ consists of all measurable
functions $f$ for which

\begin{equation*}
\left\Vert f\right\Vert _{L_{p},\infty }:=\sup_{\lambda >0} \lambda \mu
\left( f>\lambda \right)^{1/p}<+\infty .
\end{equation*}

Next, we introduce on $G_{m}$ an orthonormal system which is called Vilenkin
system.

First, we define the complex valued function $r_{k}\left( x\right) \colon
G_{m}\rightarrow \mathbb{C},$ the generalized Rademacher functions as 
\begin{equation*}
r_{k}\left( x\right) :=\exp \left( 2\pi \imath x_{k}/m_{k}\right) \quad
\left( \imath ^{2}=-1,\ x\in G_{m},\ k\in \mathbb{N}\right) .
\end{equation*}

Let us define the Vilenkin system $\psi :=(\psi _{n}:n\in \mathbb{N})$ on $%
G_{m} $ as: 
\begin{equation*}
\psi _{n}(x):=\prod_{k=0}^\infty r_{k}^{n_{k}}\left( x\right) \quad \left(
n\in \mathbb{N}\right) .
\end{equation*}

Specifically, we call this system the Walsh-Paley one if $m\equiv 2$.

The Vilenkin system is orthonormal and complete in $L_{2}\left( G_{m}\right)$
(see e.g. \cite{AVD,Vi}).

Now, we present the usual definitions in Fourier analysis.

If $f\in L_{1}\left( G_{m}\right) $ we can establish Fourier coefficients,
the partial sums of Fourier series, Dirichlet kernels with respect to the
Vilenkin system in the usual manner: 
\begin{eqnarray*}
\widehat{f}\left( k\right) &:=&\int_{G_{m}}f\overline{\psi }_{k}d\mu
,\quad\left( k\in \mathbb{N}\right) , \\
S_{n}f &:=&\sum_{k=0}^{n-1}\widehat{f}\left( k\right) \psi _{k},\quad \left(
n\in \mathbb{N}_{+},\ S_{0}f:=0\right) , \\
D_{n} &:=&\sum_{k=0}^{n-1}\psi _{k},\quad\left( n\in \mathbb{N}_{+}\right) .
\end{eqnarray*}

Recall that (see \cite{AVD}) 
\begin{equation}  \label{dn2.3}
\hspace*{0in}D_{M_{n}}\left( x\right) =\left\{ 
\begin{array}{ll}
M_{n}, & \text{ if \ }x\in I_{n}, \\ 
0, & \text{ if \ }x\notin I_{n},%
\end{array}%
\right.
\end{equation}%
and 
\begin{equation}  \label{dn2.2}
D_{n}\left( x\right) =\psi _{n}(x)\left( \underset{j=0}{\overset{\infty }{%
\sum }}D_{M_{j}}\left( x\right) \overset{m_{j}-1}{\underset{u=m_{j}-n_{j}}{%
\sum }}r_{j}^{u}\left( x\right) \right) .
\end{equation}

The $\sigma $-algebra generated by the intervals $\left\{ I_{n}\left(
x\right) :x\in G_{m}\right\} $ will be denoted by $\digamma _{n}$ $\left(
n\in \mathbb{N}\right) $. Let us denote a martingale with respect to $%
\digamma _{n}\left( n\in \mathbb{N}\right)$ by $f=\left( f_{n}:n\in \mathbb{N%
}\right) $ (for details see e.g. \cite{We1}). The maximal function of a
martingale $f$ is defined by 
\begin{equation*}
f^{\ast }=\sup_{n\in \mathbb{N}}\left\vert f^{\left( n\right) }\right\vert.
\end{equation*}

In the case $f\in L_{1}\left( G_{m}\right) ,$ the maximal function is also
given by 
\begin{equation*}
f^{\ast }\left( x\right) =\sup_{n\in \mathbb{N}}\frac{1}{\mu \left(
I_{n}\left( x\right) \right) }\left\vert \int_{I_{n}\left( x\right) }f\left(
u\right) d\mu \left( u\right) \right\vert
\end{equation*}

For $0<p<\infty $ the Hardy martingale spaces $H_{p}\left( G_{m}\right) $
consist of all martingales, for which 
\begin{equation*}
\left\Vert f\right\Vert _{H_{p}}:=\left\Vert f^{\ast }\right\Vert
_{p}<\infty .
\end{equation*}

A bounded measurable function $a$ is a $p$-atom, if there exists an interval 
$I$, such that%
\begin{equation*}
\int_{I}a d\mu =0,\quad\left\Vert a\right\Vert_{\infty }\leq \mu \left(
I\right)^{-1/p},\quad \text{supp}\left( a\right) \subseteq I.
\end{equation*}

The Hardy martingale spaces $H_{p}\left( G_{m}\right) $ have an atomic
characterization for $0<p\leq 1$. In fact the following theorem is true (see
e.g. Weisz \cite{We1,We2}):

\begin{theoremW}
A martingale $f=\left( f_{n}:n\in \mathbb{N}\right) \in H_{p}\left(
G_{m}\right)$ $\left( 0<p\leq 1\right) $ if and only if there exists a
sequence $\left( a_{k}:\ k\in \mathbb{N}\right) $ of $p$-atoms and a
sequence $\left( \mu _{k}:\ k\in \mathbb{N}\right) $ of real numbers, such
that for every $n\in \mathbb{N},$ 
\begin{equation}  \label{2A}
\sum_{k=0}^{\infty }\mu _{k}S_{M_{n}}a_{k}=f_{n}
\end{equation}%
and 
\begin{equation*}
\sum_{k=0}^{\infty }\left\vert \mu _{k}\right\vert^{p}<\infty .
\end{equation*}

Moreover, 
\begin{equation*}
\left\Vert f\right\Vert _{H_{p}}\backsim \inf \left( \sum_{k=0}^{\infty
}\left\vert \mu _{k}\right\vert^{p}\right)^{1/p},
\end{equation*}
where the infimum is taken over all decomposition of $f$ of the form %
\eqref{2A}.
\end{theoremW}

If $f\in L_{1}\left( G_{m}\right) ,$ then it is easily shown that the
sequence $\left( S_{M_{n}}f:n\in \mathbb{N}\right) $ is a martingale.

If $f=\left( f_{n},n\in \mathbb{N}\right) $ is a martingale, then
Vilenkin-Fourier coefficients are defined in a slightly different manner: 
\begin{equation*}
\widehat{f}\left( i\right) :=\lim_{k\rightarrow \infty
}\int_{G_{m}}f_{k}\left( x\right) \overline{\psi }_{i}\left( x\right) d\mu
\left( x\right) .
\end{equation*}%
Vilenkin-Fourier coefficients of $f\in L_{1}\left( G_{m}\right) $ are the
same as the martingale $\left( S_{M_{n}}f :n\in \mathbb{N}\right) $ obtained
from $f$.

\section{The Main Result}

Our main theorem reads as follows:

\begin{theorem}
\label{theorem4.4} a) Let $0<p\leq 1$ and $\left\{ \alpha _{k}:k\in \mathbb{N%
}\right\} $ be a subsequence of positive natural numbers, such that 
\begin{equation}
\sup_{k\in \mathbb{N}}\rho \left( \alpha _{k}\right) =:\varkappa <\infty .
\label{ak0}
\end{equation}%
Then the maximal operator 
\begin{equation*}
\widetilde{S}^{\ast ,\vartriangle }f:=\sup_{k\in \mathbb{N}}\left\vert
S_{\alpha _{k}}f\right\vert
\end{equation*}%
is bounded from the Hardy space $H_{p}$ to the Lebesgue space $L_{p}.$

b) Let $0<p<1$ and $\left\{ \alpha _{k}:k\in \mathbb{N}\right\} $ be a
subsequence of positive natural numbers satisfying the condition 
\begin{equation}  \label{ak}
\sup_{k\in \mathbb{N}}\rho \left( \alpha _{k}\right) =\infty .
\end{equation}
Then there exists a martingale $f\in H_{p}$ such that 
\begin{equation*}
\sup_{k\in \mathbb{N}}\left\Vert S_{\alpha _{k}}f\right\Vert _{L_{p,\infty
}}=\infty.
\end{equation*}
\end{theorem}

\begin{remark}
Since $L_{p}\subset L_{p,\infty }$ part b) means in particular that the
statement in part a) is sharp in a special sense for the case $0<p<1.$
\end{remark}

We also mention the following well-known consequences (for details see e.g.
the books \cite{We1,We2} and \cite{tep7}):

\begin{corollary}[Tephnadze \protect\cite{tep7}]
\label{corollary4.45} Let $0<p\leq 1$ and $f\in H_{p}$. Then the maximal operator 
\begin{equation*}
\sup_{n\in \mathbb{N_{+}}}\left\vert S_{M_{n}+1}f\right\vert
\end{equation*}%
is not bounded from the Hardy space $H_{p}$ to the space $L_{p}.$
\end{corollary}

In fact, we only have to notice that%
\begin{equation*}
\left\vert M_{n}+1\right\vert =n,\quad\left\langle M_{n}+1\right\rangle
=0,\quad \rho \left( M_{n}+1\right) =n.
\end{equation*}%
The second part of Theorem \ref{theorem4.4} implies our Corollary.

\begin{corollary}
\label{corollary4.44} Let $p>0$ and $f\in H_{p}$. Then the maximal operator 
\begin{equation*}
\sup_{n\in \mathbb{N_{+}}}\left\vert S_{M_{n}+M_{n-1}}f\right\vert
\end{equation*}%
is bounded from the Hardy space $H_{p}$ to the space $L_{p}.$
\end{corollary}

We notice that 
\begin{equation*}
\left\vert M_{n}+M_{n-1}\right\vert =n,\quad \left\langle
M_{n}+M_{n-1}\right\rangle =n-1, \quad \rho \left( M_{n}+M_{n-1}\right) =1.
\end{equation*}%
Thus, the second part of Theorem \ref{theorem4.4} gives again Corollary \ref%
{corollary4.44}.

\begin{corollary}
\label{corollary4.43} Let $p>0$ and $f\in H_{p}$. Then the maximal operator 
\begin{equation*}
{S}^{\#}f:=\sup_{n\in \mathbb{N}}\left\vert S_{M_{n}}f\right\vert
\end{equation*}%
is bounded from the Hardy space $H_{p}$ to the space $L_{p}.$
\end{corollary}

We find that $\left\vert M_{n}\right\vert =\left\langle M_{n}\right\rangle
=n, $ $\rho \left( M_{n}\right) =0$. Using part b) of Theorem \ref%
{theorem4.4}, we immediately get Corollary \ref{corollary4.43}.

Since $S_{n}P=P$ for every $P\in \mathcal{P}$, where $\mathcal{P}$ is the
set of all Vilenkin polynomials. The set $\mathcal{P}$ is dense in the space 
$L_{1}(G_m)$. Combining Lemma \ref{lemma2.2} and part a) of Theorem \ref%
{theorem4.4}, we obtain that under condition (\ref{ak0}) the restricted
maximal operator of partial sums is bounded from the space $L_{1}(G_m)$ to
the space $weak-L_{1}(G_m)$ It follows that

\begin{corollary}
\label{corollaryae8} Let $f\in L_{1}$ and $\left\{ \alpha _{k}:k\in \mathbb{N%
}\right\} $ be a subsequence of positive natural numbers, satisfying
condition (\ref{ak0}). Then 
\begin{equation*}
S_{\alpha _{k}}f\rightarrow f\text{\ a.e. when \ }k\rightarrow \infty .
\end{equation*}
\end{corollary}

\begin{corollary}
\label{corollaryae9} Let $f\in L_{1}$. Then%
\begin{equation*}
S_{M_{n}}f\rightarrow f\text{ \ a.e. \ when \ }n\rightarrow \infty .
\end{equation*}
\end{corollary}

\section{Lemmas}

First, we note the following well-known result, which was proved in Weisz 
\cite{We1,We2}:

\begin{lemma}
\label{lemma2.2} Suppose that an operator $T$ is sub-linear and, for some $%
0<p\leq 1$%
\begin{equation*}
\int\limits_{\overline{I}}\left\vert T a\right\vert^{p}d\mu \leq c_{p}<\infty
\end{equation*}%
{for every } $p$-atom $a$, where $I$ denotes the support of the atom $a$. If 
$T$ is bounded from $L_{\infty }$ to $L_{\infty },$ then 
\begin{equation*}
\left\Vert Tf\right\Vert _{p}\leq c_{p}\left\Vert f\right\Vert_{H_{p}}.
\end{equation*}%
Moreover, if $p<1,$ then we have weak (1,1) type estimate, i.e. it holds
that 
\begin{equation*}
\lambda \mu \left\{ x\in G_{m}:\ \left\vert Tf\left( x\right) \right\vert
>\lambda \right\} \leq \left\Vert f\right\Vert _{1}
\end{equation*}%
for all $f\in L_{1}.$
\end{lemma}

The next Lemma can be found in Tephnadze \cite{tep2}:

\begin{lemma}
\label{tep2} Let $n\in \mathbb{N}$ and $x\in I_{s}\backslash I_{s+1},$ $%
0\leq s\leq N-1$. Then 
\begin{equation*}
\int_{I_{N}}\left\vert D_{n}\left( x-t\right) \right\vert d\mu \left(
t\right) \leq \frac{cM_{s}}{M_{N}}.
\end{equation*}
\end{lemma}

We also need the following estimate of independent interest:

\begin{lemma}
\label{dn2.6.2} Let $n\in \mathbb{N}$, $\left\vert n\right\vert \neq
\left\langle n\right\rangle $ and $x\in I_{\left\langle n\right\rangle
+1}\left( e_{\left\langle n\right\rangle }\right)$ where $%
e_{k}:=(0,\dots,0,1,0,\dots)\in G_{m}$ (only the $k$-th coordinate is one,
the others are zero). Then 
\begin{equation*}
\left\vert D_{n}\left( x\right) \right\vert =\left\vert D_{n-M_{\left\vert
n\right\vert }}\left( x\right) \right\vert \geq M_{\left\langle
n\right\rangle }.
\end{equation*}
\end{lemma}

\begin{proof}
Let $x\in I_{\left\langle n\right\rangle +1}\left( e_{\left\langle
n\right\rangle }\right) $. Since 
\begin{equation*}
n=n_{\left\langle n\right\rangle }M_{\left\langle n\right\rangle
}+\sum_{j=\left\langle n\right\rangle +1}^{\left\vert n\right\vert
-1}n_{j}M_{j}+n_{\left\vert n\right\vert }M_{\left\vert n\right\vert }
\end{equation*}%
and%
\begin{equation*}
n-M_{\left\vert n\right\vert }=n_{\left\langle n\right\rangle
}M_{\left\langle n\right\rangle }+\sum_{j=\left\langle n\right\rangle
+1}^{\left\vert n\right\vert -1}n_{j}M_{j}+\left( n_{\left\vert n\right\vert
}-1\right) M_{\left\vert n\right\vert },
\end{equation*}%
Applying (\ref{dn2.3}) and (\ref{dn2.2}) we can conclude that 
\begin{eqnarray*}
\left\vert D_{n-M_{\left\vert n\right\vert }}\right\vert &\geq &\left\vert
\psi _{n-M_{n}}D_{M_{\left\langle n\right\rangle }}\sum_{s=m_{\left\langle
n\right\rangle }-n_{\left\langle n\right\rangle }}^{m_{\left\langle
n\right\rangle }-1}r_{\left\langle n\right\rangle }^{s}\right\vert
-\left\vert \psi _{n-M_{n}}\sum_{j=\left\langle n\right\rangle
+1}^{\left\vert n\right\vert
}D_{M_{j}}\sum_{s=m_{j}-n_{j}}^{m_{j}-1}r_{j}^{s}\right\vert \\
&=&\left\vert D_{M_{\left\langle n\right\rangle }}\sum_{s=m_{\left\langle
n\right\rangle }-n_{\left\langle n\right\rangle }}^{m_{\left\langle
n\right\rangle }-1}r_{\left\langle n\right\rangle }^{s}\right\vert \\
&=&\left\vert D_{M_{\left\langle n\right\rangle }}r_{\left\langle
n\right\rangle }^{m_{\left\langle n\right\rangle }-n_{\left\langle
n\right\rangle }}\sum_{s=0}^{n_{\left\langle n\right\rangle
}-1}r_{\left\langle n\right\rangle }^{s}\right\vert \\
&=&D_{M_{\left\langle n\right\rangle }}\left\vert
\sum_{s=0}^{n_{\left\langle n\right\rangle }-1}r_{\left\langle
n\right\rangle }^{s}\right\vert .
\end{eqnarray*}

Let $x_{n}=1.$ Then we readily get for $s_n<m_n$ that 
\begin{eqnarray*}
\left\vert \sum_{u=0}^{s_{n}-1}r_{n}^{u}\left( x\right) \right\vert
&=&\left\vert \frac{r_{n}^{s_{n}}\left( x\right) -1}{r_{n}\left( x\right) -1}%
\right\vert \\
&=&\frac{\sin \left( \pi s_{n}x_{n}/m_{n}\right) }{\sin \left( \pi
x_{n}/m_{n}\right) } \\
&=&\frac{\sin \left( \pi s_{n}/m_{n}\right) }{\sin \left( \pi /m_{n}\right) }%
\geq 1.
\end{eqnarray*}

It follows that 
\begin{equation*}
\left\vert D_{n-M_{\left\vert n\right\vert }}\left( x\right) \right\vert
\geq D_{M_{\left\langle n\right\rangle }}\left( x\right) \geq
M_{\left\langle n\right\rangle }.
\end{equation*}

Moreover, by using the same arguments as above it is easily seen that 
\begin{equation*}
\left\vert D_{n}\left( x\right) \right\vert =\left\vert D_{n-M_{\left\vert
n\right\vert }}\left( x\right) \right\vert ,\ \text{ for }x\in
I_{\left\langle n\right\rangle +1}\left( e_{\left\langle n\right\rangle
}\right) ,\ \left\vert n\right\vert \neq \left\langle n\right\rangle ,\ n\in 
\mathbb{N}.
\end{equation*}%
The proof is complete.
\end{proof}

\section{Proof of the main theorem}

\begin{proof}[Proof of Theorem \protect\ref{theorem4.4}]
First, we prove part a). Combining (\ref{dn2.3}) and (\ref{dn2.2}) we easily
conclude that if condition (\ref{ak0}) holds, then 
\begin{eqnarray*}
\left\Vert D_{\alpha _{k}}\right\Vert _{1} &\leq &\overset{\left\vert \alpha
_{k}\right\vert }{\underset{j=\left\langle \alpha _{k}\right\rangle }{\sum }}%
\left\Vert D_{M_{j}}\right\Vert _{1}m_{j} \\
&\leq &c\overset{\left\vert \alpha _{k}\right\vert }{\underset{%
j=\left\langle \alpha _{k}\right\rangle }{\sum }}1=c(\rho \left( \alpha
_{k}\right)+1)\leq c<\infty.
\end{eqnarray*}

It follows that $\widetilde{S}^{\ast,\vartriangle }$ is bounded from $%
L_{\infty }$ to $L_{\infty }.$ By Lemma \ref{lemma2.2} we obtain that the
proof of part a) will be complete if we show that%
\begin{equation*}
\int\limits_{\overline{I_{N}}}\left\vert \widetilde{S}^{\ast ,\vartriangle
}a\right\vert ^{p}d\mu \leq c<\infty
\end{equation*}%
for every $p$-atom $a$ with support $I=I_{N}$. Since $S_{\alpha _{k}}\left(
a\right) =0$ when $\alpha _{k}\leq M_{N},$ we can suppose that $\alpha
_{k}>M_{N}.$ (That is, $|\alpha_k|\geq N$.)

Let $t\in I_{N}$ and $x\in I_{s}\backslash I_{s+1},$ $1\leq s\leq N-1.$  If $\left\langle \alpha _{k}\right\rangle\geq N,$ we get that 
$ s<N \leq \left\langle \alpha _{k}\right\rangle $ and since $ x-t\in I_{s}\backslash I_{s+1}, $ by combining (\ref{dn2.3}) and (\ref{dn2.2}) we obtain that
\begin{equation}  \label{dnsn000}
D_{\alpha_{k}}\left(x-t\right)=0.
\end{equation}
Analogously, by combining again (\ref{dn2.3}) and (\ref{dn2.2}) we can conlude that (\ref{dnsn000}) holds, for $ s<\left\langle \alpha _{k}\right\rangle\leq N-1 $.

It follows that
\begin{equation}  \label{dnsn0}
\left\vert S_{\alpha_{k}}a\left( x\right) \right\vert =0, \ \ \text{either} \ \ \left\langle \alpha _{k}\right\rangle\geq N, \ \ \text{or} \ \ s<\left\langle \alpha _{k}\right\rangle\leq N-1.
\end{equation}

Let $0<p\leq 1$, $t\in I_{N}$ and $x\in I_{s}\backslash I_{s+1}$, $%
\left\langle \alpha _{k}\right\rangle \leq s\leq N-1.$ Applying the fact
that $\left\Vert a\right\Vert_{\infty }\leq M_{N}^{1/p}$ and Lemma \ref{tep2}
we find that 
\begin{eqnarray}  \label{dnsn}
\left\vert S_{\alpha _{k}}\left( a\right) \right\vert &\leq&
M_{N}^{1/p}\int_{I_{N}}\left\vert D_{\alpha _{k}}\left( x-t\right)
\right\vert d\mu \left( t\right) \leq c_{p}M_{N}^{1/p-1}M_{s}.  \notag
\end{eqnarray}%
Let us set $\varrho :=\min_{k\in \mathbb{N}}\left\langle \alpha
_{k}\right\rangle . $ Then, in view of (\ref{dnsn0}) and (\ref{dnsn}) we can
conclude that 
\begin{equation}  \label{dnsnM0}
\left\vert \widetilde{S}^{\ast ,\vartriangle }a\left( x\right) \right\vert
=0,\text{ \ for }\ x\in I_{s}\backslash I_{s+1},\ 0\leq s\leq \varrho
\end{equation}%
and 
\begin{equation}  \label{dnsnM}
\left\vert \widetilde{S}^{\ast ,\vartriangle }a\left( x\right) \right\vert
\leq c_{p}M_{N}^{1/p-1}M_{s},\text{ for \ }x\in I_{s}\backslash I_{s+1},\
\varrho < s\leq N-1.
\end{equation}

By the definition of $\varrho $ there exists at least one index $k_{{0}}\in 
\mathbb{N}_{+}$ such that $\varrho =\left\langle \alpha _{k_{{0}%
}}\right\rangle .$ By using contition (\ref{ak0}) we can conclude that
\begin{eqnarray}  \label{dnsnM1}
N-\varrho &=& N-\left\langle \alpha _{k_{{0}}}\right\rangle \leq \left\vert
\alpha _{k_{{0}}}\right\vert -\left\langle \alpha _{k_{{0}}}\right\rangle 
\notag \\
&\leq &\sup_{k\in \mathbb{N}}\rho \left( \alpha _{k}\right) =\varkappa
<c<\infty.  \notag
\end{eqnarray}

Let us set $m_*:=\sup_{k}m_{k}$.

Let $0<p<1.$ According to (\ref{1.1}) and using (\ref{dnsnM0}), (\ref{dnsnM}%
) and (\ref{dnsnM1}) we obtain that 
\begin{eqnarray*}
\int_{\overline{I_{N}}}\left\vert \widetilde{S}^{\ast ,\vartriangle }a\left(
x\right) \right\vert ^{p}d\mu \left( x\right) &=&\overset{N-1}{\underset{%
s=\varrho +1}{\sum }}\int_{I_{s}\backslash I_{s+1}}\left\vert \widetilde{S}%
^{\ast ,\vartriangle }a\left( x\right) \right\vert ^{p}d\mu \left( x\right)
\\
&\leq &c_{p}M_{N}^{1-p}\overset{N-1}{\underset{s=\varrho +1}{\sum }}\frac{%
M_{s}^{p}}{M_{s}} \\
&=&c_{p}M_{N}^{1-p}\overset{N-1}{\underset{s=\varrho +1}{\sum }}\frac{1}{%
M_{s}^{1-p}} \\
&\leq &\frac{c_{p}M_{N}^{1-p}}{M_{\varrho }^{1-p}} \leq c_p m_*^{\varkappa
(1-p)}\leq c_p<\infty .
\end{eqnarray*}

Let $p=1.$ We combine (\ref{dnsnM0})-(\ref{dnsnM1}) and invoke identity (\ref%
{1.1}) to obtain that 
\begin{eqnarray*}
\int_{\overline{I_{N}}}\left\vert \widetilde{S}^{\ast ,\vartriangle }a\left(
x\right) \right\vert d\mu \left( x\right) &=&\overset{N-1}{\underset{%
s=\varrho +1}{\sum }}\int_{I_{s}\backslash I_{s+1}}\left\vert \widetilde{S}%
^{\ast ,\vartriangle }a\left( x\right) \right\vert d\mu \left( x\right) \\
&\leq &c\overset{N-1}{\underset{s=\varrho +1}{\sum }}\frac{M_{s}}{M_{s}} \\
&=&c\overset{N-1}{\underset{s=\varrho +1}{\sum }}1\leq c\varkappa \leq
c<\infty .
\end{eqnarray*}

The proof of part a) is complete.

Now, we prove the second part of our theorem. Since, 
\begin{equation*}
\frac{M_{\left\vert \alpha_{k}\right\vert }}{M_{\left\langle
\alpha_{k}\right\rangle }}\geq 2^{\rho \left( \alpha _{k}\right) },
\end{equation*}
under condition (\ref{ak}), there exists an increasing subsequence $\left\{
n_{k}:\ k\in \mathbb{N}\right\} \subset \left\{ \alpha_{k}:\ k\in \mathbb{N}%
\right\} $ such that $n_{0}\geq 3$ and 
\begin{equation}  \label{16}
\lim_{k\rightarrow \infty }\frac{M_{\left\vert n_{k}\right\vert }^{\left(
1-p\right) /2}}{M_{\left\langle n_{k}\right\rangle }^{\left( 1-p\right) /2}}%
=\infty
\end{equation}%
and 
\begin{equation}
\sum_{k=0}^{\infty }\frac{M_{\left\langle n_{k}\right\rangle }^{\left(
1-p\right) /2}}{M_{\left\vert n_{k}\right\vert }^{\left( 1-p\right) /2}}%
<c<\infty .  \label{charpsn2}
\end{equation}

Let $f=\left( f_{n}:n\in \mathbb{N}\right) $ be a martingale defined by 
\begin{equation*}
f_{n}:=\sum_{\left\{ k:\ \left\vert n_{k}\right\vert <n\right\} }\lambda
_{k}a_{k},
\end{equation*}%
where 
\begin{equation*}
a_{k}:=\frac{M_{\left\vert n_{k}\right\vert }^{1/p-1}}{m_* }\left(
D_{M_{\left\vert n_{k}\right\vert +1}}-D_{M_{_{\left\vert n_{k}\right\vert
}}}\right)
\end{equation*}%
and 
\begin{equation}  \label{charpsn3}
\lambda _{k}=\frac{m_* M_{\left\langle n_{k}\right\rangle }^{\left(
1/p-1\right) /2}}{M_{\left\vert n_{k}\right\vert }^{\left( 1/p-1\right) /2}}.
\end{equation}
It is easily seen that $a$ is a $p$-atom. Under condition (\ref{charpsn2})
we can conclude that $f\in H_{p}$. (Theorem W immediately yields that $\Vert
f\Vert_{H_p}\leq c_p<\infty$.)

According to (\ref{charpsn3}) we readily see that

\begin{equation*}
\widehat{f}(j) =\left\{ 
\begin{array}{ll}
M_{\left\langle n_{k}\right\rangle }^{\left( 1/p-1\right) /2}M_{\left\vert
n_{k}\right\vert }^{\left( 1/p-1\right) /2}, & \text{\thinspace }j\in
\left\{ M_{\left\vert n_{k}\right\vert },..., M_{\left\vert n_{k}\right\vert
+1}-1\right\} ,\ k\in \mathbb{N}, \\ 
0, & j\notin \bigcup\limits_{k=0}^{\infty }\left\{ M_{\left\vert
n_{k}\right\vert },...,M_{\left\vert n_{k}\right\vert +1}-1\right\} .%
\end{array}%
\right.
\end{equation*}

Since, $M_{\left\vert n_{k}\right\vert }<n_{k}$, we get 
\begin{eqnarray*}
S_{n_{k}}f&=& \sum_{j=0}^{M_{|n_k|}-1}\hat{f}(j)\psi_j
+\sum_{j=M_{|n_k|}}^{n_k-1}\hat{f}(j)\psi_j \\
&=&S_{M_{\left\vert n_{k}\right\vert }}f+M_{\left\langle n_{k}\right\rangle
}^{\left( 1/p-1\right) /2}M_{\left\vert n_{k}\right\vert }^{\left(
1/p-1\right) /2}\psi _{M_{\left\vert n _{k}\right\vert
}}D_{n_{k}-M_{\left\vert n_{k}\right\vert }} \\
&:=&I+II.
\end{eqnarray*}

According to part a) of Theorem \ref{theorem4.4} for $I$ we have that%
\begin{equation*}
\left\Vert I\right\Vert _{L_{p},\infty}^{p}\leq \left\Vert S_{M_{\left\vert
n_{k}\right\vert }}f\right\Vert _{L_{p},\infty}^{p}\leq c_{p}\left\Vert
f\right\Vert _{H_{p}}^{p}\leq c_p<\infty .
\end{equation*}

Moreover, under condition (\ref{ak}) we can conclude that 
\begin{equation*}
\left\langle n_{k}\right\rangle \neq \left\vert n _{k}\right\vert \text{ and 
}\left\langle n_{k}-M_{\left\vert n_{k}\right\vert }\right\rangle
=\left\langle n_{k}\right\rangle .
\end{equation*}

Let $x\in I_{\left\langle n_{k}\right\rangle +1}\left( e_{\left\langle n
_{k}\right\rangle }\right) .$ Applying Lemma \ref{dn2.6.2} we obtain that 
\begin{equation*}
\left\vert D_{n_{k}-M_{\left\vert n_{k}\right\vert }}\right\vert \geq
M_{\left\langle n_{k}\right\rangle }
\end{equation*}%
Thus, we immediately have 
\begin{eqnarray*}
\left\vert II\right\vert &=&M_{\left\langle n _{k}\right\rangle }^{\left(
1/p-1\right) /2}M_{\left\vert n_{k}\right\vert }^{\left( 1/p-1\right)
/2}\left\vert D_{n_{k}-M_{\left\vert n _{k}\right\vert }}\right\vert \\
&\geq &M_{\left\langle n_{k}\right\rangle }^{\left( 1/p+1\right)
/2}M_{\left\vert n_{k}\right\vert }^{\left( 1/p-1\right) /2}.
\end{eqnarray*}

It follows that%
\begin{eqnarray*}
&&\left\Vert II\right\Vert _{L_{p,\infty }}^{p} \\
&\geq &c_{p}\left( M_{\left\langle n_{k}\right\rangle }^{\left( 1/p+1\right)
/2}M_{\left\vert n_{k}\right\vert }^{\left( 1/p-1\right) /2}\right) ^{p}\mu
\left\{ x\in G_{m}:\text{ }\left\vert II\right\vert \geq
c_{p}M_{\left\langle n_{k}\right\rangle }^{\left( 1/p+1\right)
/2}M_{\left\vert n_{k}\right\vert }^{\left( 1/p-1\right) /2}\right\} \\
&\geq &c_{p}M_{\left\vert n_{k}\right\vert }^{\left( 1-p\right)
/2}M_{\left\langle n _{k}\right\rangle }^{\left( 1+p\right) /2}\mu \left\{
I_{\left\langle n_{k}\right\rangle +1}\left( e_{\left\langle
n_{k}\right\rangle }\right) \right\} \geq \frac{c_{p}M_{\left\vert
n_{k}\right\vert }^{\left( 1-p\right) /2}}{M_{\left\langle
n_{k}\right\rangle }^{\left( 1-p\right) /2}}.
\end{eqnarray*}

Hence, for large enough $k$, 
\begin{eqnarray*}
&&\left\Vert S_{n_{k}}f\right\Vert_{L_{p,\infty }}^{p} \\
&\geq &\left\Vert II\right\Vert_{L_{p,\infty }}^{p}-\left\Vert I\right\Vert
_{L_{p,\infty }}^{p}\geq \frac{1}{2}\left\Vert II\right\Vert _{L_{p,\infty
}}^{p} \\
&\geq &\frac{c_{p}M_{\left\vert n_{k}\right\vert }^{\left( 1-p\right) /2}}{%
M_{\left\langle n_{k}\right\rangle }^{\left( 1-p\right) /2}}\rightarrow
\infty ,\ \text{ when }\ k\rightarrow \infty .
\end{eqnarray*}

The proof is complete.
\end{proof}

\end{document}